\documentclass[english]{article}
\usepackage[utf8]{inputenc}
\usepackage[T1]{fontenc}
\usepackage{babel}
\usepackage{amsmath, amssymb, amsthm}
\usepackage{graphicx}
\usepackage{fancyhdr}
\usepackage{verbatim}
\newtheorem{proposition}{Proposition}
\newtheorem{lemma}{Lemma}
\newtheorem{theorem}{Theorem}

\theoremstyle{definition} \newtheorem*{nots}{Notations}
\theoremstyle{definition} 
\theoremstyle{definition} 
\theoremstyle{definition} \newtheorem*{def1}{Definition}
\theoremstyle{definition} 
\theoremstyle{definition} \newtheorem{cor}{Corollary}

\begin{document}

\title{\bf Moment subset sums over finite fields}

\author{Tim Lai\textsuperscript{1}, Alicia Marino\textsuperscript{2}, Angela Robinson\textsuperscript{3}, Daqing Wan\textsuperscript{4}}

\maketitle
\thispagestyle{fancy}

\begin{center}
\textsuperscript{1}Indiana University, Bloomington\\
\textsuperscript{2}University of Hartford\\
\textsuperscript{3}National Institute of Standards and Technology\\
\textsuperscript{4}University of California, Irvine

\end{center}

\vspace{1cm}

{\bf Abstract:} The $k$-subset sum problem over finite fields is a classical NP-complete problem.
Motivated by coding theory applications, a more complex problem is the higher $m$-th moment $k$-subset sum problem over finite fields. We show that there is a deterministic polynomial time algorithm for the $m$-th moment $k$-subset sum problem over finite fields for each fixed $m$ when the evaluation set is the image set 
of a monomial or Dickson polynomial of any degree $n$. In the 
classical case $m=1$, 
this recovers previous results 
of Nguyen-Wang (the case $m=1, p>2$) \cite{WN18} and the results of Choe-Choe (the case $m=1, p=2$) \cite{CC19}. 


\section{Introduction}
One of the most puzzling problems in theoretical computer science, originally posed in 1971, is to determine whether P = NP \cite{C71}.  That is, to determine whether the complexity class of problems which can be solved in deterministic polynomial time is equivalent to the class of problems whose solutions, if any, can be verified in deterministic polynomial time. For a comprehensive survey on this topic, see Widgerson's forthcoming monograph \cite{Wi19}. 

All NP-complete problems are equivalent to each other under polynomial time reduction. 
One approach to proving that P = NP is to find an NP-complete problem and prove (or disprove) that it is deterministically solvable in polynomial time.  
We choose the $k$-subset sum problem over finite fields \cite{CLRS09}, which is 
a classical NP-complete problem. Although this problem is out of reach, our aim of this paper is to explore deterministic polynomial time algorithms to this and similar variations of this problem in various interesting special cases.  

Let $p$ be a prime, $q=p^s$ for some integer $s>0$, and $\mathbb{F}_q$ the finite field of $q$ elements.  Given a subset $D=\{x_1, \ldots, x_d\} \subset \mathbb{F}_q$ and $b \in \mathbb{F}_q$, let
\begin{align*}
N(D,b)= \#\{S \subseteq D: \sum_{x \in S}x=b\}.
\end{align*}
The dense input size of $D$ is $d\log q$, since one can simply list all the $d$ elements of $D$ in $\mathbb{F}_q$ where each takes $\log q$ space. 
The decision subset sum problem (SSP) over finite fields asks if given $D$ and $b$, can one determine whether $N(D,b)>0$ in polynomial time in terms of the dense input size $d\log q$?  If $N(D,b)>0$, then there exists at least one collection $S \subseteq D$ of elements which sum to $b$.  This solution, $S$, can be checked by addition of $|S|\leq d$ elements of size $\textnormal{log }q$, thus SSP $\in$ NP for every fixed $p$.  When $p=2$, it is a linear algebra problem and thus 
SSP $\in$ P.  It is known SSP is NP-complete for each fixed $p>2$.  

Motivated by numerous applications, a more precise version of the SSP is to determine whether there exists a subset $S \subseteq D$ of given size $k$ whose elements sum to $b$ given a set $D$ and target $b$ as above.  The decision version of this $k$-subset sum problem ($k$-SSP) is as follows. Given a subset $D=\{x_1, \ldots, x_d\} \subset \mathbb{F}_q, k\in \{1, \ldots, d\}$ and $b \in \mathbb{F}_q$, for
\begin{align*}
N_k(D,b)= \#\{S \subseteq D: \sum_{x \in S}x=b, |S|=k\},
\end{align*}  
determine whether $N_k(D,b)>0$.  
The decision $k$-SSP problem is NP-hard for every fixed $p$, including the 
more difficult case $p=2$ which is the main result in \cite{V97} determining that computing the 
minimum distance of binary codes is NP-hard. In general, the complexity of the $k$-SSP problem depends on the relationship between $d$ and the modulus $q$.  When $q=\mathcal{O}(\textnormal{poly}(d))$, dynamic programming solves the problem in polynomial time \cite{GM91, L05}. The trivial exhaustive search algorithm shows that $k$-SSP $\in$ P when $d = \mathcal{O}(\textnormal{log log}\,q)$. It is known that $k$-SSP is NP-hard when $d=(\textnormal{log }q)^c$ for constant $c>0$, see \cite{LO85,GM91}.  An explicit formula for $N_k(D,b)$ was presented for the case of $D = \mathbb{F}_q$ \cite{LW08}.





In coding theory, $k$-SSP arises from computing the minimum distance of a linear code and the deep hole problem for  Reed-Solomon codes. The set $D$ is called the \textit{evaluation set} as it is exactly the evaluation set of the corresponding Reed-Solomon code. 
If one moves further to consider the harder problem of computing the error distance 
of a received word (namely, maximal likelihood decoding) in Reed-Solomon codes, one is naturally lead to 
the following higher moment $k$-subset sum problem. 
More formally,  given a subset $D=\{x_1, \ldots, x_d\} \subset \mathbb{F}_q, k\in \{1, \ldots, d\}$, $m \in \mathbb{N}$, and $\boldsymbol{b} = (b_1, \ldots, b_m) \in \mathbb{F}_q^m$, determine whether
\begin{align*}
N_k(D,\boldsymbol{b},m)= \#\{S \subseteq D: \sum_{y \in S}y^j=b_j, 1\leq j \leq m, |S|=k\},
\end{align*}  
is positive. This problem is known as the $m$-th moment $k$-SSP and 
its complexity has been studied recently. It is proven to be NP-hard for general $D$ if $m\leq 3$ \cite{GGG15} or smaller than $\mathcal{O}(\textnormal{log}\,\textnormal{log}\,\textnormal{log}\,q)$ \cite{GGG18}. An explicit combinatorial 
formula for $N_k(D, \boldsymbol{b}, m)$ is obtained 
in \cite{Ng19} when $m=2$ and $D=\mathbb{F}_q$. 

All the problems and results above are based on a model where we use the dense input $\{D,b\}$ of size $\mathcal{O}(d\, \textnormal{log}\,q)$ by listing all the $d$ elements of $D$.  Though improved solutions to the decision $k$-SSP with such dense input are desired, one may also consider an \textit{algebraic input} model wherein $D$ is the set of images under some polynomial map applied to field elements.  That is, for some monic polynomial $g(x) \in \mathbb{F}_q [x]$ of degree $n$,
\[
	D = g(\mathbb{F}_q) = \{g(a) : a \in \mathbb{F}_q\}.
\]
In this situation, the algebraic input size would be $n\log q$ since it is enough to write 
down the $n$ coefficients of the input polynomial $g(x)$. A fundamental problem is to ask 
if the $k$-SSP and the $m$-th moment $k$-SSP can be solved in deterministic polynomial time 
in terms of the algebraic input size $n\log q$. This appears more difficult as it is not 
even clear if the problem is in NP because both $k$ and the set size $d=|D| \geq q/n$ can already be 
exponential in terms of the algebraic input size $n\log q$. No complexity result is yet known for the algebraic model. 

The last author conjectured that $k$-SSP can be solved in deterministic polynomial time in 
algebraic input size $n\log q$  if the order of the 
Galois group $G_g$ of $g(x)-t$ over ${\mathbb{F}}_q(t)$ is bounded by 
a polynomial in $n\log q$.  The last condition is trivially satisfied if 
$$n = O(\log\log q/\log\log\log q)$$ 
since then $|G_g| \leq n!$ is bounded by a polynomial in $\log q$. 
This condition is also satisfied when $g(x)$ is a monomial 
or 
Dickson polynomial of any degree $n$.  
Note that this conjecture is highly non-trivial, as it is not even clear whether the problem is in NP since 
we are using the algebraic (sparse) input 
size and $d\geq q/n$ is exponential in 
$n\log q$ for 
$n=O(\log\log q)$. Thus, we cannot write down all the elements of $D$ as listing all 
elements of $D$ already takes exponential time. In a sense, our set $D$ is given as 
a black-box. 

As a supporting evidence, this conjecture has been proved to be 
true in the special case when the evaluation set $D$ is the image of the monomial 
$x^n$ or Dickson polynomials of degree $n$: see \cite{WN18} for the case $p>2$ and 
\cite{CC19} for the 
case $p=2$. The aim of the present paper is to extend these results from $m=1$ ($k$-SSP) 
to the higher $m$-moment $k$-SSP for each fixed $m$. Namely, our main result is 
\begin{theorem}\label{THM1}
Let the evaluation set $D$ be the image set of a monomial or a Dickson polynomial of degree $n$ 
over $\mathbb{F}_q$. 
There is a deterministic algorithm which for 
any given $m\in \mathbb{N}$, $b\in \mathbb{F}_q^m$ and integer 
$k\geq 0$, 
decides if $N_k(D,b,m)>0$ 
in time $(n\log q)^{C_m}$, where $C_m$ is a constant depending only on $m$. 
In particular, this is a polynomial time algorithm in the algebraic input 
size $n\log q$ for each fixed $m$. 

\end{theorem}

To prove the above theorem, 
we will need to combine all the 
techniques available so far:  dynamic programming for large $n>q^{\epsilon}$, 
Kayal's algorithm \cite{K05} for 
constant $k$, Brun sieve for medium $k$, the Li-Wan sieve 
for large $k$ and $p>2$, and the recent Choe-Choe argument \cite{CC19} for large $k$ and $p=2$. 
In addition, we need to employ the Weil bound to prove a 
crucial new partial character 
sum estimate. 


\section{Background} 
One important tool in our proof is character sum estimates.  Let $\psi: \mathbb{F}_q \rightarrow \mathbb{C}$ be an additive character. 
We know from character theory that for a nontrivial character $\psi$ we have $\sum\limits_{x \in \mathbb{F}_q} \psi (x) = 0$. However, in the case of the trivial character, the sum is the size of the finite field. 

Let $G = \mathbb{F}_q$ and let $\widehat{G}$ be the set of all additive characters for $\mathbb{F}_q$. Then we have the following equality
\[
\sum_{\psi \in \widehat{G}} \psi(x) = \left\{
        \begin{array}{ll}
            q & \quad \textnormal{if }x = \textbf{0} \\[1em]
            0 & \quad \textnormal{if }x \neq \textbf{0}
        \end{array}
    \right..
\]

%
%


\begin{def1}[Dickson Polynomial]
Let $n$ be a positive integer and $a\in \mathbb{F}_q$. The Dickson polynomial of degree $n$ is defined as
\begin{equation*}
D_n(x, a) = \sum_{i=0}^{\lfloor n/2 \rfloor} \frac{n}{n-i}\binom{n-i}{i} (-a)^i x^{n-2i}. 
\end{equation*}
\end{def1}

\noindent If $n=pn_1$ is divisible by $p$, one checks that $D_{pn_1}(x, a) = D_{n_1}(x, a)^p$. Thus, we can assume that $n$ is not divisible by $p$. 

Note that for $a=0$, $D_n(x,0)=x^n$, so we see that Dickson polynomials are generalizations of monomials. Of particular use to us is the size of the image of these polynomials, also known as the \textit{value set}. A simple fact for the monomial $D_n(x,0)=x^n$ is that 
\[
 |D_n(\mathbb{F}_q^\times,0)| = 
 \begin{cases} 
      q-1 & \gcd(n, q-1)=1 \\
      \frac{1}{\ell}(q-1) &  \gcd(n, q-1)=\ell
   \end{cases}
\]
 In the first case, the map is $1$ to $1$; in the latter case, the map is $\ell$ to $1$. It turns out an analogous preimage-counting statement holds when $a\ne 0$. Chou, Mullen, and Wassermann in \cite{CMW} used a character sum argument to calculate the following. 

\begin{nots} For $b,c,d \in \mathbb{Z}$, Let $b^c||d$ denote that $b^c$ fully divides $d$ so that $b^{c+1} \nmid d$.  
\end{nots}

\begin{theorem}\label{dicksonTheorem}
Let $n\geq 2$ and $a\in \mathbb{F}_q^*$. If $q$ is even, then $|D_n^{-1}(D_n(x_0,a))| =$
\begin{equation*}
 \left\{
     \begin{array}{cl}
       \gcd(n, q-1) & \text{ if condition A holds} \\
       \gcd(n, q+1) & \text{ if condition B holds} \\
       \dfrac{\gcd(n,q-1)+\gcd(n,q+1)}{2} & \text{$D_n(x_0,a)=0$}, 
     \end{array}
   \right .
\end{equation*}
where `condition A' holds $\text{if } x^2+x_0x+a \text{ is reducible over $\mathbb{F}_q$ and } D_n(x_0,a)\ne  0$; `condition B' holds $\text{if } x^2+x_0x+a \text{ is irreducible over $\mathbb{F}_q$ and } D_n(x_0,a)\ne 0$. \newline

\noindent If $q$ is odd, let $\eta$ be the quadratic character of $\mathbb{F}_q$. If $2^r||(q^2-1)$ then $|D_n^{-1}(D_n(x_0,a))|=$
\begin{equation*}
\left\{
     \begin{array}{cl}
       \gcd(n, q-1) & \text{if } \eta(x_0^2-4a)=1 \text{ and } D_n(x_0,a)\ne \pm 2a^{n/2} \\
       \gcd(n, q+1) & \text{if } \eta(x_0^2-4a)=-1 \text{ and } D_n(x_0,a)\ne \pm 2a^{n/2} \\
       \dfrac{\gcd(n, q-1)}{2} & \text{if } \eta(x_0^2-4a)=1 \text{ and condition C holds} \\
       \dfrac{\gcd(n, q+1)}{2} & \text{if } \eta(x_0^2-4a)=-1 \text{ and condition C holds} \\
       \dfrac{\gcd(n,q-1)+\gcd(n,q+1)}{2} & \text{otherwise}, \\
     \end{array}
   \right.
\end{equation*}
where `condition C' holds if
\begin{equation*}
2^t||n \text{ with } 1\leq t\leq r-1, \eta(a)=-1, \text{ and } D_n(x_0,a)=\pm 2a^{n/2}
\end{equation*}
or
\begin{equation*}
2^t||n \text{ with } 1\leq t\leq r-2, \eta(a)=1, \text{ and } D_n(x_0,a)=- 2a^{n/2}.
\end{equation*}
\end{theorem}

\noindent They also showed an explicit formula for the size of the value set of $D_n(x,a)$, denoted $|V_{D_n(x,a)}|$. 
We state their result in the odd $q$ case. 

\begin{theorem}\label{dicksonTheorem2}
Let $a\in \mathbb{F}_q^*$. If $2^r||(q^2-1)$ and $\eta$ is the quadratic character on $\mathbb{F}_q$ when $q$ is odd, then
\begin{equation*}
|V_{D_n(x,a)}| = \frac{q-1}{2\gcd(n,q-1)}+\frac{q+1}{2\gcd(n,q+1)}+\delta
\end{equation*}
where
\begin{equation*}
\delta = \left\{
     \begin{array}{ll}
       1 & \text{if $q$ is odd, $2^{r-1}||n$ and $\eta(a)=-1$} \\
       \dfrac{1}{2} & \text{if $q$ is odd, $2^{t}||n$ with $1\leq t\leq r-2$} \\
       0 & \text{otherwise}.
     \end{array}
    \right.
\end{equation*}
\end{theorem}

As a consequence, for Dickson polynomials 
of degree $n$, the value set cardinality $d=|D|$ can be computed in polynomial time in $n\log q$. Note that for a general 
polynomial $g(x)\in \mathbb{F}_q[x]$ of 
degree $n$, computing the image size 
$|g(\mathbb{F}_q)|$ is a difficult problem, and there is no known polynomial time algorithm in terms of the algebraic 
input size $n\log q$, see \cite{CHW13} 
for complexity results and $p$-adic algorithm.

\subsubsection*{Weil's Character Sum Bound}

The following classical case of the Weil bound is well known. We shall give a more general 
form later. 

\begin{theorem} (Weil Bound)
Let $f(x) \in \mathbb{F}_q[x]$ be a polynomial of degree $m$, where $(p,m) = 1$ and $\psi$ a non-trivial additive character of $\mathbb{F}_q$. Then
\[
\left | \sum_{x \in \mathbb{F}_q} \psi ( f(x)) \right | \le (m-1) \sqrt{q}.
\]
\end{theorem}

For our purposes it will be important to have a good estimate for certain incomplete character sums, where the sum 
is not summing over the full field $\mathbb{F}_q$, but over the image set $D$ of another polynomial $g(x)$. This is not available yet for general $g(x)$, but can be proved for monomials and Dickson polynomials. The monomial case is straightforward. 


\begin{proposition}\label{MonomialWeil}
Let $f(x) \in \mathbb{F}_q[x]$ be a polynomial of degree $m$ such that $p \nmid m$. 
Let $D = \{x^n : x \in \mathbb{F}_q\}$ where $(n+1)^2 \leq q$. Then 
\[
\left | \sum_{x \in D} \psi ( f(x)) \right | \le m \sqrt{q}.
\]
\end{proposition}

\begin{proof}
Without loss of generality, we can assume that $n|(q-1)$.  
Let $D^\times = \{x^n : x \in \mathbb{F}_q^\times \}$. Using the Weil bound above,
\begin{align*}
    \left| \sum_{x \in D} \psi(f(x)) \right| &= \left| \psi(f(0)) +\sum_{x \in D^\times} \psi(f(x)) \right| \\
    &= \left| \psi(f(0)) + \frac{1}{n} \sum_{x \in \mathbb{F}_q^\times} \psi(f(x^n)) \right| \\
    &= \left| \psi(f(0)) + \frac{1}{n} \sum_{x \in \mathbb{F}_q} \psi(f(x^n)) -\frac{1}{n} \psi(f(0))\right| \\
    &\leq 1 + \frac{1}{n}(mn-1)\sqrt{q} + \frac{1}{n} \\
    &= 1 + m\sqrt{q} - \frac{\sqrt{q}-1}{n}.
\end{align*}
If $(n+1)^2 \leq q$ then we conclude $$\left | \sum_{x \in D} \psi ( f(x)) \right | \le m \sqrt{q}.$$
\end{proof}

When $D$ is the image of Dickson polynomials, the corresponding character sum estimate is harder. We need the following version of Weil's bound, which is the case $d=1$ of Theorem 5.6 in \cite{FW}.

\begin{theorem}
Let $f_i(t)$ ($1\leq i\leq n$) be polynomials in $\mathbb{F}_q[t]$, let $f_{n+1}(t)$ be a rational function in $\mathbb{F}_q(t)$, let $D_1$ be the degree of the highest square free divisor of $\prod_{i=1}^n f_i(t)$, let
\[D_2= \begin{cases} 
      0 & \deg(f_{n+1})\leq 0 \\
      \deg(f_{n+1}) & \deg(f_{n+1})>0,
   \end{cases}
\]
let $D_3$ be the degree of the denominator of $f_{n+1}$, and let $D_4$ be the degree of the highest square free divisor of the denominator of $f_{n+1}(t)$ which is relatively prime to $\prod_{i=1}^n f_i(t)$. 

Let $\chi_i:\mathbb{F}_q^* \to \mathbb{C}^*$ $(1\leq i\leq n)$ be multiplicative characters of $\mathbb{F}_q$, and let $\psi=\psi_p \circ \text{Tr}_{\mathbb{F}_q / \mathbb{F}_p}$ for a non-trivial additive character $\psi_p: \mathbb{F}_p \to \mathbb{C}^*$ of $\mathbb{F}_p$. Extend $\chi_i$ to $\mathbb{F}_q$ by setting $\chi_i(0)=0$. Suppose that $f_{n+1}(t)$ is not of the form $r(t)^p-r(t)+c$ in $\mathbb{F}_q(t)$. Then, we have 
\begin{align*}
&\left| \sum_{a\in \mathbb{F}_{q}, f_{n+1}(a)\ne \infty} \chi_1(f_1(a)) \cdots \chi_n(f_n(a)) \psi(f_{n+1}(a)) \right| \\
&\phantom{\sum}\leq (D_1 + D_2 + D_3 + D_4 - 1)\sqrt{q}, 
\end{align*}
where the sum is taken over those $a\in \mathbb{F}_{q}$ such that $f_{n+1}(a)$ is well-defined.
\end{theorem}

\noindent As a consequence, we derive the following character sum bounds.

\begin{cor}\label{weilBounds}
Let $\psi_{\text{Tr}}=\psi_p \circ \text{Tr}_{\mathbb{F}_q / \mathbb{F}_p}$ be the 
canonical additive character, 
$\psi:\mathbb{F}_q\to \mathbb{C}^*$ any non-trivial additive character of $\mathbb{F}_q$, and $\eta:\mathbb{F}_q^*\to \mathbb{C}^*$ the quadratic character if $q$ is odd. Let $f(x)$ be a polynomial in $\mathbb{F}_q[x]$ of degree $m$ not divisible by $p$.\begin{enumerate}
\item For all $q$, we have  
\begin{equation*}
\left| \sum_{x\in \mathbb{F}_q} \psi(f(D_n(x,a)))\right| \leq (mn-1)\sqrt{q}. 
\end{equation*}
\item If $q$ is odd, then 
\begin{equation*}
\left| \sum_{x\in \mathbb{F}_q} \eta(x^2-4a)\psi(f(D_n(x,a)))\right| \leq (mn+1)\sqrt{q}. 
\end{equation*}
\item If $q$ is even, then 
\begin{align*}
\left|\sum_{x\in \mathbb{F}_q^*} \psi_{\text{Tr}}\left( f(D_n(x,a))+a/x^2\right)\right| &=  \left| \sum_{x\in \mathbb{F}_q^*} \psi_{\text{Tr}}\left( f(D_n(x,a))+a^{q/2}/x\right)  \right| \\
&\leq (mn+1)\sqrt{q}. 
\end{align*}
\end{enumerate}

Note that none of the polynomials in place of $f_{n+1}(x)$ are of the form $r(t)^2-r(t)+c$. This is clear if $n$ is also not divisible by $p$. 
If $n$ is divisible by $p$, it can be reduced to the case when $n$ is not divisible by $p$ using the identity $D_{pn_1}(x, a) = D_{n_1}(x, a)^p$. 

The following lemma is the key character sum estimate we need. The proof follows the method used in \cite{KW16}, where the case $m=1$ is treated. 

\end{cor}

\begin{lemma}\label{dicksonWeilEven} Let $f(x)$ be a polynomial in $\mathbb{F}_q[x]$ of degree $m$ not divisible by $p$. Let $D=\{D_n(x,a)\ |\ x\in \mathbb{F}_q\}$, for $a\in \mathbb{F}_q^*$. If $\psi:(\mathbb{F}_q,+)\to \mathbb{C}^*$ is a non-trivial additive character, then the following estimates hold:
\begin{equation*}
\left|\sum_{x\in D} \psi(f(x)) \right|\leq (mn+1)\sqrt{q}. 
\end{equation*}
\end{lemma}

\begin{proof}
The sum can be rewritten in the following way:
\begin{equation*}
S_f:= \sum_{y\in D} \psi(f(y)) = \sum_{x\in \mathbb{F}_q} \psi(f(D_n(x,a))) \frac{1}{N_x}, 
\end{equation*}
where $N_x=|D_n^{-1}(D_n(x,a))|$ is size of the preimage of the value $D_n(x,a)$. 

\noindent \textbf{When $q$ is even:}

\noindent By Theorem \ref{dicksonTheorem}, $N_x$ can be quantified. Let $\text{Tr}:\mathbb{F}_q\to \mathbb{F}_2$ denote the absolute trace. Using the fact that $z^2+xz+a$ is reducible over $\mathbb{F}_q$ if and only if $\text{Tr}(a/x^2)=0$, we obtain 
\begin{align*}
S_f &=\sum_{\substack{x\in \mathbb{F}_q^* \\
\text{Tr}(a/x^2)=0}} \frac{1}{\gcd(n, q-1)} \psi(f(D_n(x,a))) \\
&+ \sum_{\substack{x\in \mathbb{F}_q^* \\ \text{Tr}(a/x^2)=1}} \frac{1}{\gcd(n, q+1)} \psi(f(D_n(x,a))) \\
&+ \frac{1}{\gcd(n, q-1)}\psi(f(D_n(0,a))) + O(1), 
\end{align*}
where $O(1)$ is a constant of size at most 1, which we accept by dropping the $D_n(x,a)=0$ case. Denote $\psi_1: \mathbb{F}_2\to \mathbb{C}^*$ as the order two additive character and $\psi_{\text{Tr}}=\psi_1\circ \text{Tr}$, which is an additive character from $\mathbb{F}_q\to \mathbb{C}^*$. Simplifying and rearranging gives
\begin{align*}
S_f &= \frac{1}{2\gcd(n,q-1)} \sum_{x\in \mathbb{F}_q^*} \psi(f(D_n(x,a)))(1+\psi_{\text{Tr}}(a/x^2))  \\ 
&\phantom{=}+ \frac{1}{2\gcd(n, q+1)} \sum_{x\in \mathbb{F}_q^*} \psi(f(D_n(x,a)))(1-\psi_{\text{Tr}}(a/x^2)) \\
&\phantom{=}+ \frac{1}{\gcd(n, q-1)}\psi(D_n(0,a)) + O(1) \\
&= \left( \frac{1}{2\gcd(n,q-1)} + \frac{1}{2\gcd(n, q+1)} \right) \sum_{x\in \mathbb{F}_q^*} \psi(f(D_n(x,a))) \\ 
&\phantom{=}+ \left( \frac{1}{2\gcd(n,q-1)} - \frac{1}{2\gcd(n, q+1)} \right) \sum_{x\in \mathbb{F}_q^*} \psi(f(D_n(x,a)))\psi_{\text{Tr}}(a/x^2) \\
&\phantom{=}+ \frac{1}{\gcd(n, q-1)}\psi(f(D_n(0,a))) + O(1). 
\end{align*}
We add and subtract $\left(\frac{1}{2\gcd(n,q-1)} + \frac{1}{2\gcd(n, q+1)} \right)\psi(D_n(0,a))$ to complete the first sum:
\begin{align*}
&= \left( \frac{1}{2\gcd(n,q-1)} + \frac{1}{2\gcd(n, q+1)} \right) \sum_{x\in \mathbb{F}_q} \psi(f(D_n(x,a))) \\
&+ \left( \frac{1}{2\gcd(n,q-1)} - \frac{1}{2\gcd(n, q+1)} \right) \sum_{x\in \mathbb{F}_q^*} \psi(f(D_n(x,a)))\psi_{\text{Tr}}(a/x^2) \\
&+ \left( \frac{1}{2\gcd(n,q-1)} -\frac{1}{2\gcd(n, q+1)} \right)\psi(f(D_n(0,a))) + O(1). 
\end{align*}
In order to estimate the sum in second term, take $b\in \mathbb{F}_q^*$ so that $\psi(x)=\psi_{\text{Tr}}(bx)$. Then,
\begin{equation*}
\sum_{x\in \mathbb{F}_q^*} \psi(f(D_n(x,a)))\psi_{\text{Tr}}(a/x^2) = \sum_{x\in \mathbb{F}_q^*} \psi_{\text{Tr}}(bf(D_n(x,a))+a/x^2). 
\end{equation*}
Applying the bounds in Corollary \ref{weilBounds} with $f$ replaced by $bf$,
\begin{align*}
\left|\sum_{y\in D} \psi(f(y))\right| &\leq \left( \frac{1}{2\gcd(n,q-1)} + \frac{1}{2\gcd(n, q+1)} \right)(mn-1)\sqrt{q} \\
&\phantom{=}\phantom{=}+ \left| \frac{1}{2\gcd(n,q-1)} - \frac{1}{2\gcd(n, q+1)} \right| (mn+1)\sqrt{q} + 2 \\
&\leq (mn+1)\sqrt{q}.
\end{align*}

\noindent \textbf{When $q$ is odd:}

\noindent We use Theorem \ref{dicksonTheorem} again to calculate $N_x$. Let $\eta$ be the quadratic character of $\mathbb{F}_q$. Then, 
\begin{align*}
S_f &=\sum_{\substack{x\in \mathbb{F}_q \\ \eta(x^2-4a)=1}} \frac{1}{\gcd(n, q-1)} \psi(f(D_n(x,a))) \\
&+ \sum_{\substack{x\in \mathbb{F}_q \\ \eta(x^2-4a)=-1}} \frac{1}{\gcd(n, q+1)} \psi(f(D_n(x,a))) + O(1). 
\end{align*}
The term $O(1)$ is a constant of size at most 2, which we accept by dropping the complicated `condition C' and `otherwise' cases. Simplifying and rearranging gives
\begin{align*}
&= \frac{1}{2\gcd(n,q-1)} \sum_{x\in \mathbb{F}_q} \psi(f(D_n(x,a)))(1+\eta(x^2-4a))  \\ 
&\phantom{=}+ \frac{1}{2\gcd(n, q+1)} \sum_{x\in \mathbb{F}_q} \psi(f(D_n(x,a)))(1-\eta(x^2-4a)) + O(1) \\
&= \left( \frac{1}{2\gcd(n,q-1)} + \frac{1}{2\gcd(n, q+1)} \right) \sum_{x\in \mathbb{F}_q} \psi(f(D_n(x,a))) \\ 
&\phantom{=}+ \left( \frac{1}{2\gcd(n,q-1)} - \frac{1}{2\gcd(n, q+1)} \right) \sum_{x\in \mathbb{F}_q} \psi(f(D_n(x,a)))\eta(x^2-4a) + O(1). 
\end{align*}
Again applying the bounds in Corollary \ref{weilBounds},
\begin{align*}
\left|\sum_{x\in D} \psi(f(x)) \right|&\leq \left( \frac{1}{2\gcd(n,q-1)} + \frac{1}{2\gcd(n, q+1)} \right)(mn-1)\sqrt{q} \\ &\phantom{=}\phantom{=}+ \left| \frac{1}{2\gcd(n,q-1)} - \frac{1}{2\gcd(n, q+1)} \right|(mn+1)\sqrt{q} + 2\\
&\leq (mn+1)\sqrt{q}, 
\end{align*}
which was to be shown.
\end{proof}

\section{$k$-MSS($m$)}
We are now ready to consider the $m$-th moment $k$-subset sum problem, called $k$-MSS($m$) in short. Let $m$ be a fixed positive integer, and  $g(x) \in \mathbb{F}_q [x]$ a polynomial of degree $n$ with $1\leq n \leq q-1$. 
Let $D = g ( \mathbb{F}_q)$ 
and $\boldsymbol{b} = (b_1, b_2, \ldots, b_m) \in \mathbb{F}_q^m$. Since we are working in characteristic $p$, we have that $$(x_1^i + \ldots + x_k^i)^p = x_1^{ip} + \ldots + x_k^{ip}.$$ Thus if $b_i^p \neq b_{ip}$ for some $ip\leq m$, there will be no solutions for $k$-MSS($m$). We may and will assume without loss of generality that $b_i^p  = b_{ip}$ for all $ip\leq m$ in the remainder of this paper. Under this assumption, the $j$-th power equation in the $k$-MSS($m$) can and will be dropped for all $j$ divisible by $p$.  
We introduce the moment subset sum problem over subsets of size $k$ with the value

\begin{align}
N_k(D,\boldsymbol{b},m)=\# \left\{ S \subseteq D : |S| = k, \sum_{y \in S} y^j = b_j , 1 \le j \le m, p \nmid j \right\}.
\end{align}
Thus, from now on, the index $j$ is
not divisible by $p$. 

Determining whether $N_k(D,\boldsymbol{b},m) >0$ for given $\{D, \boldsymbol{b}\}$ is the decision version of the $k$-MSS($m$) problem. 
As indicated before, we shall use the 
algebraic input size $n\log q$. 


A closely related number is the following integer
\begin{align*}
M_k (D,\boldsymbol{b},m) = \#  \{ (x_1, \dots, x_k) \in D^k : \sum_{i=1}^k x_i^{j} = b_j, \\
x_{i_1} \ne x_{i_2}, \forall \,1 \le i_1 < i_2 \le k, \ p \nmid j  \}.
\end{align*}
It is clear that $M_k (D,\boldsymbol{b},m) = k! N_k (D,\boldsymbol{b},m)$. 
We deduce 
\begin{theorem} $M_k(D,\boldsymbol{b},m) > 0$ if and only if $N_k(D,\boldsymbol{b},m) > 0$. 
\end{theorem}
Our problem is then reduced to deciding if $M_k (D,\boldsymbol{b},m)>0$. 
We can reduce this further by assuming from duality that $k \le \frac{|D|}{2}$. The strategy to solve this new problem is to combine all established strategy for the original subset sum problem and apply the character sum estimate from the previous section. We shall divide $k$ into three different ranges (constant size, 
medium size, and large size) and 
use different methods for each range. 
The main idea is to use algorithms to solve boundary cases of parameters and 
to use mathematics to prove that 
there is a solution when the parameters 
are in the interior. 

If $n>q^{\epsilon}$ for constant $\epsilon>0$, then $q$ is polynomial in 
$n\log q$, we can list all elements of $D$
and use the dynamic programming algorithm 
to solve the moment subset sum problem in polynomial time. 
In the rest of the paper, we can and will 
assume that $n < q^{\epsilon}$ for whatever positive constant $\epsilon$ we like.

\subsection*{$k$-MSS($m$) for constant size $k$}
The main result that we depend on in this case is due to Kayal's solvability algorithm for polynomial systems over $\mathbb{F}_q$ \cite{K05}, which we summarize in this context below.
Let $f_1, \dots, f_m \in \mathbb{F}_q[x_1, \dots, x_n]$, where $d$ is the maximum degree of all the polynomials. Let $X = V(f_1, \dots, f_m)$ be the vanishing locus of the polynomials. Then the result of Kayal \cite{K05} states the following. 
\begin{theorem}
The decision problem of $\# X(\mathbb{F}_q) > 0$ can be solved in time $\left(d^{n^{cn}} m \log q\right)^{O(1)}$ for some constant $c>0$. 
\end{theorem}
Most of the conditions in our $k$-MSS($m$)  are polynomial equations, with the exception of the condition that the individual elements be distinct. However, we can easily consider this as a polynomial equation at the cost of additional variables. Recall that $D =\{g(x) : x \in \mathbb{F}_q\, , g(x) \in \mathbb{F}_q [x]\}$ for a polynomial $g$ such that $\text{deg}(g)=n$. For the context of the $k$-MSS($m$) problem,  we are deciding if the variety determined by the vanishing locus of 
$$f_j (x_1, \dots, x_k ):=\left(\sum_{i=1}^k g(x_i)^j\right) -b_j, \ 1 \le j \le m, \ p\nmid j$$
and the additional polynomial

\[
\left( \prod_{i_1 \ne i_2}  (g(x_{i_1}) - g(x_{i_2}))   \right)x_{k+1}  - 1
\]
have any $\mathbb{F}_q$-rational points. Each $f_j$ has degree at most $mn$ while the latter polynomial has degree $n \binom{k}{2} + 1$.

Now, we assume $k \le 3m+1$. 
Then, $n\binom{k}{2} + 1 \leq 9nm^2$ and so all the polynomials have degrees 
bounded by $9nm^2$. 
Kayal's theorem then states that the decision problem can be solved in time which is bounded by a 
polynomial in 
\[
    (9nm^2)^{(k+1)^{O((k+1))}}\log q = (9nm^2)^{(3m+2)^{O((3m+2))}}\log q.
\]
This is  $(n \log q)^{O(1)}$ if $m$ is a constant. Thus, we have proved the following
\begin{theorem}\label{ThmSmallK}
Let $D =\{g(x) : x \in \mathbb{F}_q\}$, where $g(x) \in \mathbb{F}_q [x]$ is any polynomial of degree $n$. Let 
$m$ be a fixed positive integer. Assume $k \leq 3m+1$. Then $k$-MSS(m) can be solved in time $(n \log q)^{O(1)}$. 
\end{theorem}
The condition $k\leq 3m+1$ is all we need. It can be replaced by 
any bound $k\leq C$, where $C$ is a positive constant. 

\subsection*{$k$-MSS($m$) for medium $k$}


We now consider the moment $k$-subset sum problem for medium-sized values of $k$.  
Fix $m \in \mathbb{N}$ and $\boldsymbol{b} = (b_1, \ldots, b_m) \in \mathbb{F}_q^m$. Let $m_p = |\{j : 1 \leq j \leq m, p \nmid j\}| = m - \lfloor \frac{m}{p} \rfloor$. 
Recall 
\begin{align*}
  M_k(D,\boldsymbol{b},m) = |\{(x_1, \ldots, x_k ) \in D^k :& \sum_{i=1}^k x_i^{j} - b_j =0, \\
  &x_{i_1} \neq x_{i_2} \text{ for } i_1 \neq i_2, \\
  &1 \leq j \leq m, p \nmid j\}|.
\end{align*}
and 
\begin{align*}
M_k(D,\boldsymbol{b},m) = k! \cdot N_k(D,\boldsymbol{b},m),
\end{align*}
where $$N_k(D,\boldsymbol{b},m) = |\{S \subseteq D : |S| = k, \sum\limits_{y \in S} y^j = b_j, 1 \leq j \leq m, p \nmid j\}|.$$
We wish to decide when 
$M_k(D,\boldsymbol{b},m) > 0$. 
The following theorem solves this problem 
in the medium $k$ case if certain character sum estimate is satisfied. 

\begin{theorem}\label{med_k_thm}
Let $D = g(\mathbb{F}_q)$ where $g \in \mathbb{F}_q[x]$ with deg$(g) = n$.  Let $\psi$ be a non-trivial 
additive character of $\mathbb{F}_q$.
Assume for all $f \in \mathbb{F}_q[x]$ of degree at most $m$ with $p \nmid \text{deg}(f)$, 
we have 
$$\left|\sum_{x \in D} \psi(f(x))\right| \leq (mn+1)\sqrt{q}.$$ 
Then 
$M_k(D,\boldsymbol{b},m) > 0$ 
if $2n(mn+1) < q^\frac{1}{6}$ and $3m_p+1 < k < q^\frac{5}{12}$. 
\end{theorem}

The first condition $2n(mn+1) < q^\frac{1}{6}$ is already satisfied, since we assumed that $n<q^{\epsilon}$ and $m$ is a constant. 
The second condition $3m_p+1 < k < q^\frac{5}{12}$ gives the medium range of $k$. 

Towards this goal, we define
$$R= |\{(x_1, \ldots, x_k ) \in D^k : \sum_{i=1}^k x_i^{j} - b_j = 0, 1 \leq j \leq m, p \nmid j\}|.$$

We say that $\boldsymbol{x}=(x_1, \ldots, x_k )\in \mathbb{F}_q^k$ \textit{is a solution} if $\boldsymbol{x}$ satisfies the conditions of $R$. Note that $R$ counts solutions allowing for those with repeated entries, while $M_k(D,\boldsymbol{b},m)$ strictly counts solutions with  distinct entries.  We define a new number to compute the size of $R$ with the added condition that the first two entries of $\boldsymbol{x}$ are equal.  Let


$$R_{12} = |\{(x_1, \ldots, x_k ) \in D^k : 2x_2^{j} +\sum_{i=3}^k x_i^{j} - b_j = 0, 1 \leq j \leq m, p \nmid j\}|.$$

Then the Brun sieve tells us that 
$$M_k(D,\boldsymbol{b},m) \geq R - \sum_{1 \leq i_1 < i_2 \leq k} R_{i_1 i_2} = R - \binom{k}{2}R_{12}.$$
In order to rewrite $R$ and $R_{12}$ and obtain bounds for them we use the theory of characters.

Let $\psi$ be a non-trivial additive character of $\mathbb{F}_q$. 
Recall that we have the following summation:
\[
\sum_{c \in \mathbb{F}_q} \psi(cx) = \left\{
        \begin{array}{ll}
            q & \quad \text{if } x = 0 \\[1em]
            0 & \quad \text{if } x \neq 0
        \end{array}
    \right.
\]
We would like to take advantage of this character sum equation and have it evaluate solutions positively and evaluate non-solutions to zero. Thus we have the following identity.
\[
\prod_{j=1, p\nmid j}^m\left(\sum_{c \in \mathbb{F}_q} \psi(c(\sum_{i=1}^k x_i^j-b_j))\right) = \left\{
        \begin{array}{ll}
          q^{m_p} & \quad \text{if } \textbf{$x$} \text{ is a solution}\\[1em]
          0 & \quad \text{if } \textbf{$x$} \text{ is not a solution}
        \end{array}
    \right.
\]
With this in mind, we can rewrite $R$ as below
\begin{align*}
    R &= \frac{1}{q^{m_p}}\sum_{x \in D^k}\prod_{j=1, p\nmid j}^m\sum_{c \in \mathbb{F}_q} \psi(c(\sum_{i=1}^k x_i^j-b_j))\\
    &= \frac{1}{q^{m_p}} \sum_{\boldsymbol{x} \in D^k} \sum_{\boldsymbol{c} \in \mathbb{F}_q^{m_p}} \prod_{j=1, p\nmid j}^m \psi(c_j (\sum_{i=1}^k x_i^j -b_j)) \\
    &= \frac{1}{q^{m_p}} \sum_{\boldsymbol{c} \in \mathbb{F}_q^{m_p}}\sum_{\boldsymbol{x} \in D^k} \prod_{j=1, p\nmid j}^m \psi(c_j (\sum_{i=1}^k x_i^j -b_j)) \\
    &= \frac{1}{q^{m_p}} \sum_{\boldsymbol{c} \in \mathbb{F}_q^{m_p}}\sum_{\boldsymbol{x} \in D^k} \psi\left(\sum_{j=1, p\nmid j}^m c_j \left(\sum_{i=1}^k x_i^j -b_j\right)\right)
\end{align*}
By separating the contribution of the trivial term, we obtain the following.
\begin{align*}
    R &= \frac{1}{q^{m_p}} \sum_{\boldsymbol{x} \in D^k} \psi(0) + \frac{1}{q^{m_p}}\sum_{0 \neq \boldsymbol{c} \in \mathbb{F}_q^{m_p}} \sum_{\boldsymbol{x} \in D^k} \psi\left(\sum_{j=1, p\nmid j}^m c_j \left(\sum_{i=1}^k x_i^{j}-b_j\right)\right) \\
    &= \frac{|D|^k}{q^{m_p}} + \frac{1}{q^{m_p}} \sum_{0 \neq \boldsymbol{c} \in \mathbb{F}_q^{m_p}} S_c,
\end{align*}
where
$$S_c = \sum_{\boldsymbol{x} \in D^k} \psi\left(\sum_{j=1, p\nmid j}^m c_j \left(\sum_{i=1}^k x_i^{j}-b_j\right)\right).$$ 
Define 
$$f(x) = \sum\limits_{j=1, p\nmid j}^m c_j x^{j} \in \mathbb{F}_q[x].$$
Note that the degree of $f$ is not divisible by $p$ and at most $m$ if $c\not=0$. We now want to find an upper bound for $S_c$. Notice that 
\begin{align*}
    \psi\left(\sum_{j=1, p\nmid j}^m c_j \left(\sum_{i=1}^k x_i^{j}-b_j\right)\right) &= \psi\left(\sum_{i=1}^k\sum_{j=1, p\nmid j}^m c_j x_i^{j} - \sum_{j=1, p\nmid j}^m c_jb_j\right) \\
    &= \psi(f(x_1)) \cdots \psi(f(x_k))\psi(-\sum_{j=1, p\nmid j}^m c_jb_j) \\
    &= A \cdot \prod_{i=1}^k \psi(f(x_i)).
\end{align*}
Here, $A = \psi(-\sum\limits_{j=1, p\nmid j}^m c_jb_j)$ and so $|A| = \prod\limits_{j=1, p\nmid j}^m|\psi(-c_jb_j)|=1$. Thus $$|S_c| = \left|\sum\limits_{\boldsymbol{x} \in D^k}\prod\limits_{i=1}^k \psi(f(x_i))\right| = \left(\left|\sum\limits_{x \in D} \psi(f(x))\right|\right)^k.$$
By our assumptions, $|S_c| \leq (mn+1)^k(\sqrt{q})^k$. It follows that 
$$\left| R - \frac{|D|^k}{q^{m_p}} \right| = \frac{1}{q^{m_p}}\sum_{0 \neq \boldsymbol{c} \in \mathbb{F}_q^{m_p}} |S_c| \leq \frac{q^{m_p}-1}{q^{m_p}}(mn+1)^k q^\frac{k}{2} <(mn+1)^k q^\frac{k}{2}.$$

{\bf Remark}. Igor Shparlinski kindly informed us that the average trick in \cite{Sh15} can be used to improve 
the above coefficient $(mn+1)^k$ to $(mn+1)^{k-2}$. 
The idea is to apply the character sum estimate only to the first  $(k-2)$-th power in $|S_c|$, and then compute the remaining quadratic moment over $c$, resulting in a saving of the factor 
$(mn+1)^2$. This type of improvement is theoretically 
interesting, but would not significantly improve the lower bound condition $3m_p+1 <k$ in our theorem, 
which is enough for our algorithmic purpose of 
this paper.

Now we can rewrite $R_{12}$ in a similar way.
\begin{align*}
    R_{12} &= \frac{1}{q^{m_p}} \sum_{\boldsymbol{x} \in D^{k-1}} \prod_{j=1, p\nmid j}^m \sum_{c \in \mathbb{F}_q} \psi(c(2x_1^j + \sum_{i=3}^k x_i^j-b_j)) \\
    &= \frac{1}{q^{m_p}} \sum_{\boldsymbol{x} \in D^{k-1}} \sum_{\boldsymbol{c} \in \mathbb{F}_q^{m_p}} \prod_{j=1, p\nmid j}^m \psi(c_j (2x_1^j + \sum_{i=3}^k x_i^j -b_j)) \\
    &= \frac{1}{q^{m_p}} \sum_{\boldsymbol{c} \in \mathbb{F}_q^{m_p}}\sum_{\boldsymbol{x} \in D^{k-1}} \prod_{j=1, p\nmid j}^m \psi(c_j (2x_1^j + \sum_{i=3}^k x_i^j -b_j)) \\
    &= \frac{1}{q^{m_p}} \sum_{\boldsymbol{c} \in \mathbb{F}_q^{m_p}}\sum_{\boldsymbol{x} \in D^{k-1}} \psi\left(\sum_{j=1, p\nmid j}^m c_j \left(2x_1^j + \sum_{i=3}^k x_i^j -b_j\right)\right)
\end{align*}
By separating the contribution of the trivial character, we obtain the following.
\begin{align*}
    R_{12} &= \frac{1}{q^{m_p}} \sum_{\boldsymbol{x} \in D^{k-1}} \psi(0) + \frac{1}{q^{m_p}}\sum_{0 \neq \boldsymbol{c} \in \mathbb{F}_q^{m_p}} \sum_{\boldsymbol{x} \in D^{k-1}} \psi\left(\sum_{j=1, p\nmid j}^m c_j \left(2x_1^j + \sum_{i=3}^k x_i^j -b_j\right)\right) \\
    &= \frac{|D|^{k-1}}{q^{m_p}} + \frac{1}{q^{m_p}} \sum_{0 \neq \boldsymbol{c} \in \mathbb{F}_q^{m_p}} S_c^{12},
\end{align*}
where $$S_c^{12} = \sum\limits_{\boldsymbol{x} \in D^{k-1}}\psi(\sum\limits_{j=1, p\nmid j}^m c_j (2x_1^{j} + \sum\limits_{i=3}^{k} x_i^{j}-b_j)).$$ By a similar manipulation in the previous case,
\begin{align*}
S_c^{12} &= \sum\limits_{\boldsymbol{x} \in D^{k-1}}\psi(2f(x_1))\psi(f(x_3))\cdots\psi(f(x_{k}))\psi(-\sum\limits_{j=1, p\nmid j}^m c_jb_j) \\
&= A\sum\limits_{\boldsymbol{x} \in D^{k-1}}\psi(2f(x_1))\prod_{i=3}^{k}\psi(f(x_i)).
\end{align*}
By a rearrangement, we see that 
\begin{align*}
    \left|S_c^{12}\right| &= \left| \sum_{\boldsymbol{x} \in D^{k-1}} \psi(2f(x_1))(\prod_{i=3}^{k} \psi(f(x_i)))\right| \\
    &= \left| \left(\sum_{x \in D} \psi(2f(x))\right)\left(\sum_{x \in D} \psi(f(x))\right)^{k-2}\right|
\end{align*}
By our assumptions, if $p>2$ (and thus $2\not=0$), 
\begin{align*}
    \left|S_c^{12}\right| &\leq (mn+1)\sqrt{q}(mn+1)^{k-2}(\sqrt{q})^{k-2} \\
    &= (mn+1)^{k-1}q^\frac{k-1}{2}.
\end{align*}
The case $p=2$ can be handled in a similar way, and one get the alternate bound 
$$\left|S_c^{12}\right| \leq 
    |D|(mn+1)^{k-2}q^\frac{k-2}{2}.$$
We assume that $p>2$ for simplicity. 
Now we have that
\begin{align*}
    \left|R_{12} - \frac{|D|^{k-1}}{q^{m_p}}\right| &= \frac{1}{q^{m_p}} \left|\sum_{\textbf{0} \neq \textbf{c} \in \mathbb{F}_q^{m_p}} S_c^{12}\right| \\
    &\leq \frac{1}{q^{m_p}}\sum_{\textbf{0} \neq \textbf{c} \in \mathbb{F}_q^{m_p}} (mn+1)^{k-1}q^\frac{k-1}{2} \\
    &= \frac{q^{m_p}-1}{q^{m_p}} (mn+1)^{k-1}q^\frac{k-1}{2} \\
    &< (mn+1)^{k-1}q^\frac{k-1}{2}.
\end{align*}
Since we have the following two inequalities,  
$$\left|R_{12} - \frac{|D|^{k-1}}{q^{m_p}}\right| < (mn+1)^{k-1}q^\frac{k-1}{2}$$ 
$$\left|R - \frac{|D|^{k}}{q^{m_p}}\right| < (mn+1)^k q^\frac{k}{2}$$ 
we see that 
$$\frac{|D|^k}{q^{m_p}}-(mn+1)^k q^\frac{k}{2} < R, \text{ and }$$ 
$$R_{12} < \frac{|D|^{k-1}}{q^{m_p}} + (mn+1)^{k-1}q^\frac{k-1}{2}.$$ 
Then 
\begin{align*}
    R &- \binom{k}{2} R_{12} > \frac{|D|^{k}}{q^{m_p}}-(mn+1)^k q^\frac{k}{2} - \binom{k}{2}\left(\frac{|D|^{k-1}}{q^{m_p}} + (mn+1)^{k-1}q^\frac{k-1}{2} \right) \\
    &= |D|^{k-1}\frac{1}{q^{m_p}}\left(|D|- \binom{k}{2}\right)-(mn+1)^{k-1}q^\frac{k-1}{2}\left((mn+1)\sqrt{q} + \binom{k}{2}\right) \\
    &= \frac{1}{q^{m_p}}\left(|D|^{k-1}\left(|D|-\binom{k}{2}\right)\right) - (mn+1)^{k-1}q^\frac{k-1}{2}\left((mn+1)\sqrt{q} + \binom{k}{2}\right).
\end{align*}
We wish to show that $R- \binom{k}{2} R_{12}$ is positive and thus we need to show that $$|D|^{k-1}\left(|D|-\binom{k}{2}\right) \geq q^{m_p}(mn+1)^{k-1}q^\frac{k-1}{2}\left((mn+1)\sqrt{q} + \binom{k}{2}\right).$$
However since $\text{deg}(g) = n$ we know that $|D| \geq \frac{q}{n}$. Thus it is enough to show that $$\left(\frac{q}{n}\right)^{k-1}\left(\frac{q}{n}-\binom{k}{2}\right) \geq q^{m_p+\frac{k-1}{2}}(mn+1)^{k-1}\left((mn+1)\sqrt{q} + \binom{k}{2}\right).$$
Towards this goal, we utilize our assumptions that $2n(mn+1) < q^\frac{1}{6}$ and $3m_p+1 < k < q^\frac{5}{12}$.
It is enough to prove 
$$\left(\frac{q}{n}\right)^{k-1} \geq q^{m_p+\frac{k-1}{2}}(mn+1)^{k-1}, \ 
\left(\frac{q}{n}-\binom{k}{2}\right) \geq \left((mn+1)\sqrt{q} + \binom{k}{2}\right). 
$$
For the first inequality, we have 
\begin{align*}
    \left(\frac{q}{n}\right)^{k-1}>q^{m_p+\frac{k-1}{2}}(mn+1)^{k-1} 
    &\iff q^{k-1-m_p-\frac{k-1}{2}} > (mn+1)^{k-1}n^{k-1} \\
    &\iff q^{\frac{k-1}{2}-m_p} > (mn+1)^{k-1}n^{k-1}.  \\
\end{align*}
Since $2n(mn+1) < q^\frac{1}{6}$, 
the right side is bounded by 
$$(mn+1)^{k-1}n^{k-1}<(n(mn+1))^{k-1} < q^\frac{k-1}{6}.$$
Our problem is now reduced to showing that $q^{\frac{k-1}{6}+m_p} < q^{\frac{k-1}{2}}$. Namely, 
$$m_p < \frac{k-1}{2}-\frac{k-1}{6} =\frac{k-1}{3}.$$
This is satisfied since $3m_p+1 < k$. 
Thus we have shown that 
\begin{equation}\label{med_proof_eq_1}
    \left(\frac{q}{n}\right)^{k-1}>q^{m_p+\frac{k-1}{2}}(mn+1)^{k-1}.
\end{equation}

For the second inequality, we need to show that $n(mn+1)\sqrt{q} + 2n\binom{k}{2} < q$. Since $k < q^\frac{5}{12}$ and  $2n(mn+1) < q^\frac{1}{6}$, 
we know that $k^2n < q^\frac{5}{6}q^\frac{1}{6}/2 = q/2$. 
We deduce that 
\begin{equation}\label{med_proof_eq_2}
    n(mn+1)\sqrt{q} + 2n\binom{k}{2} < \frac{q^{1/6 +1/2}}{2} + \frac{q}{2} < q.
\end{equation}
The theorem is proved. 


\begin{cor}
Let $D = \{x^d : x \in \mathbb{F}_q\}$ or $D = \{D_n(x,a) : x \in \mathbb{F}_q\}$ for $a \in \mathbb{F}_q^\times$.  Then $M_k(D,b,m) > 0$ if $2n(mn+1) < q^\frac{1}{6}$ and $3m_p+1 < k < q^\frac{5}{12}$. 
\end{cor}

Let $\psi$ be a non-trivial additive character of $\mathbb{F}_q$. 
We have shown that  all $f \in \mathbb{F}_q[x]$ of degree at most $m$ with $p \nmid \text{deg}(f)$,
$$\left|\sum_{x \in D} \psi(f(x))\right| \leq m \sqrt{q}$$ 
if $D = \{x^d : x \in \mathbb{F}_q\}$, and 
$$|\sum_{x \in D} \psi(f(x))| \leq (mn+1) \sqrt{q}$$ 
if $D = \{D_n(x,a) : x \in \mathbb{F}_q\}$. 
Since $m\sqrt{q} \leq (mn+1)\sqrt{q}$, the character sum condition in Theorem \ref{med_k_thm} is satisfied. The medium case is proved. 

\subsection*{$k$-MSS($m$) for large $k$}
Following established procedures, we use the Li-Wan sieve \cite{LW10} to analyze large values of $k$. This method has been used several times \cite{ZW12, KW16, LW10, LW18, LW19,WN18} and is now standard. So, we will only give an outline and indicate the differences. 
We begin by discussing the relevant notation and concepts that we will apply in our context. In this section, we assume that $D$ is the image of a monomial or Dickson polynomial of degree $n$. The 
relevant character sum estimate is then true. 

We use the notation $S_k$ to denote the symmetric group on $k$ letters. For a permutation $\tau \in S_k$, its disjoint cycle decomposition is written as 
\[
\tau = (a_1 a_2 \cdots a_{m_1})(a_{m_1+1} \cdots a_{m_2}) \cdots (a_{m_{k-1}+1} \cdots a_{m_k}).
\]
We shall refer to $\tau$ interchangeably with its disjoint cycle decomposition, which we fix beforehand.

Denote by $\overline{X} = \{(x_1,\dots,x_k) \in D^k : x_i \ne x_j, \forall i \ne j\}$. For the sake of brevity, we will denote $k$-tuples from such products by $x = (x_1, \dots, x_k)$ when there is no risk of confusion. 
Let $\psi$ be a fixed non-trivial additive character of $\mathbb{F}_q$.
Recall from earlier sections that we are interested in
\[
h_c(x_1,\dots,x_k) = \psi \left( 
\sum_{j=1,p\nmid j}^m   c_j \left(\sum_{i=1}^k x_i^{j}-b_j\right)\right),
\]
where $c$ is not the zero vector. 
Now define 
$$F(c) = \sum\limits_{x \in \overline{X}} h_c(x_1,\dots,x_k), \ F_\tau(c) = \sum\limits_{x \in X_\tau} h_c(x_1, \dots x_k),$$ 
where $X_\tau$  consists of tuples in $\overline{X}$ such that
\begin{align*}
x_{a_1} = \cdots = x_{a_{m_1}}, x_{m_1+1} = \ldots = x_{m_2}, \ldots, x_{m_{k-1}+1} = \ldots = x_{m_k}
\end{align*}
and so on. Now, let's think of $\tau$ as having $e_1$ cycles of length $1$, $e_2$ cycles of length $2$, and so on, up until $e_k$ cycles of length $k$. Note that $\sum\limits_{i=1}^k ie_i = k$. This allows us to express $F_\tau(c)$ as:
\begin{align*}
F_\tau(c) &= \sum_{x \in X_\tau} \psi \left( \sum_{j=1,p\nmid j}^m c_j \left( \sum_{i=1}^k i (x_{i1}^{j} + \cdots + x_{ie_i}^{j})  - b_j \right)\right)   \\
&= \sum_{\substack{x_{il} \in D \\ 1 \le i \le k\\ 1 \le l \le e_i}} \psi \left( \sum_{j=1,p\nmid j}^m c_j \left( \sum_{i=1}^k \sum_{l=1}^{e_i} i x_{il}^{j} \right) \right)\psi ( \sum_{j=1,p\nmid j}^m -c_j b_j) \\
&= \sum_{\substack{x_{il} \in D \\ 1 \le i \le k\\ 1 \le l \le e_i}} \prod_{i=1}^k \psi^i \left( \sum_{p\nmid j, l} c_j x_{il}^{j} \right) \psi ( \sum_{j=1,p\nmid j}^m -c_j b_j).
\end{align*}
Let's consider the inner sum. 
\[
\sum_{x_{il}\in D}\psi^i(\sum_{p\nmid j} c_j x_{il}^{j}) = \sum_{x \in D} \psi ^i (f(x))
\]
where $f(x) = \sum_{j=1,p\nmid j}^m c_j x^{j}$. Hence, 
if the $c_j$'s are not all zero, we have
\[
\left |\sum_{x\in D} \psi(\sum_{j=1,p\nmid j}^m c_j x^{j}) \right | \le (mn+1) \sqrt{q}.
\]
Now the order of $\psi$ is $p$ so the order of $\psi^i$ is $\frac{p}{(i,p)}$, which is $p$ unless $p \mid i$, in which case it is $1$. Therefore,
\begin{align*}
|F_\tau(c)| &= \left | \prod_{i=1}^k \left( \sum_{x \in D} \psi^i \left( \sum_{j=1,p\nmid j}^m c_j x^{j} \right) \right) ^{e_i}  \psi(\sum_{j=1,p\nmid j}^m -c_jb_j) \right| \\
& \le \prod_{\substack{i \\
1 \le i \le k\\
p \nmid i}} ((mn+1) \sqrt{q})^{e_i} \cdot \prod_{\substack{i \\
1 \le i \le k\\
p \mid i}}  |D|^{e_i}. 
\end{align*}

The Li-Wan sieve says that 
\[
F(c) = \sum_{\sum ie_i = k} (-1)^{k - \sum e_i} N(e_1, \dots, e_k) F_{e_1, \dots, e_k}(c),
\]
where $N(e_1,..., e_k)$ 
denote the number of permutations in $S_k$ 
with cycle type $(e_1,..., e_k)$, and $F_{e_1, \dots, e_k}(c)$ denotes $F_{\tau}(c)$ for any $\tau$ of cycle type 
$(e_1,..., e_k)$. 
Using the above estimates and Lemma 2.1 in \cite{WN18}, one obtains 
\begin{align*}
|F(c)| &\le \sum_{\sum ie_i = k} N(e_1, \dots, e_k) \prod_{(i,p) =1} ((mn+1) \sqrt{q})^{e_i} \cdot \prod_{p \mid i} |D|^{e_i} \\
& \le \left ((mn+1) \sqrt{q} + k + \frac{|(mn+1)\sqrt{q} - |D||}{p} -1 \right )_k
\end{align*}
where we define $(x)_k := x(x-1) \cdots (x-k+1)$.

This concludes our discussion of the Li-Wan sieve and the appropriate adaptation to our context. We now return to the framework in the previous sections, with notations as before. Let's see how the above Li-Wan helps.
Recall 
\begin{align*}
M_k(D,b,m) &= \sum_{x \in \textbf{X}} \frac{1}{q^{m_p}} \sum_\psi \sum_{c_j \in \mathbf{F}_q} \psi \left( \sum_{j=1,p\nmid j}^m c_j \left( \sum_{i=1}^k x_i^{j} - b_j \right) \right) \\
&= \frac{1}{q^{m_p}} (|D|)_k + \sum_{(\cdots c_j \cdots) \ne 0} \frac{1}{q^{m_p}}F(c).
\end{align*}
Therefore,
\begin{align}
\left | M_k(D,b,m) - \frac{1}{q^{m_p}} (|D|)_k \right | &< \left((mn+1) \sqrt{q} + k + \frac{\left |(mn+1) \sqrt{q} - |D| \right|}{p} - 1\right)_k \\
& \le \left(0.013 |D| + k + \frac{|D|}{p} \right)_k.
\end{align}
This estimate is the analogue of equation (2.3) in \cite{WN18}, resulting from assuming that 
\[
(mn+1)\sqrt{q} \le 0.013 |D|.
\]
If further, $6m_p \ln q \leq k \leq \frac{|D|}{2}$, the same argument as in 
the proof of Theorem 2.3 in \cite{WN18} 
shows that $M_k(D, b, m)>0$. 
We obtain 

\begin{theorem}
 Let $D$ be the image of a monomial of Dickson polynomial of degree $n$ . 
 Assume that $p>2$, $(mn+1)\sqrt{q} \leq 0.013|D|$, and $6m_p \ln q \leq k \leq \frac{|D|}{2}$. Then, $M_k(D,b,m)>0$. 
\end{theorem}

Note that if $p=2$, the same proof works, but only for $k$  in the shorter range 
$6m_p \ln q \leq k \leq \frac{(1-\epsilon)|D|}{2}$. That is, $k$ cannot reach 
all the way to $|D|/2$ if $p=2$.  

Since $|D|\geq q/n$, the condition 
$(mn+1)\sqrt{q} \leq 0.013|D|$ is satisfied if $n(mn+1)\sqrt{q} \leq 0.013q$, which is certainly true since 
$m$ is fixed and $n<q^{\epsilon}$.

\begin{section}{Case $p=2$}
Finally, we examine the $k$-MSS($m)$ over finite fields of characteristic 2.  The result of Kayal used for $k$-MSS($m$) for constant $k$ and our proof for medium-sized $k$ still hold in fields of characteristic 2. Thus Theorem \ref{ThmSmallK} and Theorem \ref{med_k_thm} hold for $q=p^s$ for all $p$.

To analyze the case $p=2$ for large $k$, we rely on recent work by Choe and 
Choe \cite{CC19} which examines the subset sum problem over finite fields of characteristic 2 .  We adjust the definitions of this work to fit the higher moment subset sum problem over $D$ which are images of monomials or Dickson polynomials. Note that $p=2$ in this section. 

We will prove an analogue of Theorem 2.3 in \cite{CC19}. Let $D \subseteq \mathbb{F}_q$, $k \leq |D|/2$, and $f(x) = \sum\limits_{j=1, p\nmid j}^m c_j x^{j}$, for $c_j \in \mathbb{F}_q$.  For a nontrivial additive character $\psi$ of $\mathbb{F}_q$, define

\begin{align*}
    S_D(k,\psi,f)=\sum_{\substack{x_{i} \in D \\
x_i \ \textnormal{distinct}}}\psi(f(x_1)+f(x_2)+ \ldots +f(x_k)).
\end{align*}

Although $S_D(k,\psi,f)$ sums over distinct $x_i$, there is no assumption that the $f(x_i)$ are distinct.  Over finite fields of characteristic 2, however, if $x_i=x_j$, then $f(x_i)=f(x_j)$, and the sum $f(x_i)+f(x_j)$ is equivalent to $2f(x_i)=0$. It follows that  
\begin{align*}
    S_D(2,\psi,f)&=\sum_{\substack{x_{1},x_{2} \in D \\
x_1 \neq x_2}}\psi(f(x_1)+f(x_2)) \\
    &= (\sum_{x \in D}\psi(f(x)))^2 -|D|. 
\end{align*}
By induction , one derives the following recursive formula for $S_D(k,\psi,f)$ for all $k>1$, which is the analogue of Lemma 2.1 \cite{CC19}. 

\renewcommand\labelitemi{$\cdot$} 

\begin{lemma}
Let $D$  be a subset of $\mathbb{F}_q$ with more than 3 elements and $\psi$ be a nontrivial additive character of $\mathbb{F}_q$.  Then 
\begin{itemize}
 \item $S_D(1,\psi,f)=\sum\limits_{x\in D}\psi(f(x))$, 
 \item $S_D(2,\psi,f)= S_D(1,\psi,f)^2-|D|$, and 
 \item $S_D(k,\psi,f)= S_D(1,\psi,f)S_D(k-1,\psi,f)-(|D|-k+2)(k-1)S_D(k-2,\psi,f)$, where $3 \leq k \leq |D|$.
\end{itemize} 
\end{lemma}
This lemma can be applied to prove analogue of Lemma 2.2 \cite{CC19}.  The statement is as follows.
\begin{lemma}
Let $D$ be a subset of $\mathbb{F}_q$ with more than 4 elements and $\psi$ be a nontrivial additive character of $\mathbb{F}_q$.  If 
\begin{align*}
    \left|\sum_{x \in D}\psi(f(x))\right| \leq \frac{1}{16}|D|, 
\end{align*}
then
\begin{align*}
    |S_D(k,\psi,f)| < \left(\frac{9}{16}|D|\right)^k, \textnormal{ for all } k \leq \frac{|D|}{2}.
\end{align*}
\end{lemma}

From Proposition \ref{MonomialWeil} and Lemma \ref{dicksonWeilEven}, it follows that when $D$ is the image of a polynomial of degree $n$ such that the value set character sum estimate satisfies
\begin{align*}
    \left|\sum_{x \in D}\psi(f(x))\right| < (mn+1)\sqrt{q},
\end{align*}
then the condition $n(mn+1) < \frac{1}{16}\sqrt{q}$ implies that
\begin{align*}
  \left|\sum_{x \in D}\psi(f(x))\right| <  (mn+1)\sqrt{q} < \frac{1}{16}\frac{q}{n} \leq \frac{1}{16}|D|.
\end{align*}
As in the previous section, a standard character sum argument gives the inequality 
\begin{align}
   \left|M_k(D,b,m) - (\frac{1}{q})^{m_p} (|D|)_k \right| < \max_{c\in \mathbb{F}_q^{m_p}-0} 
   S(k, \psi, f_c), 
\end{align}
where $f_c = \sum_{j=1, p\nmid j}^m c_j x^{j}$. It follows that 
\begin{align}
   \left|M_k(D,b,m) - (\frac{1}{q})^{m_p} (|D|)_k \right| < \left(\frac{9}{16}|D|\right)^k.
\end{align}
The same argument as in the proof of Theorem 2.3 in \cite{CC19} shows that 
if 
$$3.05sm_p=3.05 m_p\log_2q < k \leq |D|/2,$$
then 
\begin{align}
    \frac{1}{q^{m_p}} (|D|)_k > \frac{1}{q^{m_p}}\left(\frac{9}{16}|D|\right)^k 2^{sm_p}
                        = \left(\frac{9}{16}|D|\right)^k,
\end{align}
Thus, we obtain
\begin{theorem} Let $p=2$ and $n(mn+1) < \frac{1}{16}\sqrt{q}$.
Then $M_k(D,b,m)>0$ for all 
$3.05 m_p\log_2q < k \leq |D|/2$.
\end{theorem}
We conclude that when $D$ is the image of degree $n$
polynomial satisfying the value set character sum estimate in Lemma \ref{dicksonWeilEven}, the $m$-th moment
subset sum problem over $D$ can be solved in
deterministic polynomial time in the algebraic
input size $n \log q$, for every constant $m$. In particular, this is true when $D$ is the image of a monomial of Dickson polynomial of degree $n$.

\end{section}

\begin{section}{Conclusion}
We show that there is a deterministic polynomial time algorithm for the $m$-th moment $k$-subset sum problem over finite fields for each fixed $m$ when the evaluation set is the image set 
of a monomial or Dickson polynomial of any degree $n$.
An open problem is to ask if Theorem \ref{THM1} can be proved 
for larger range of $m$, say, $m=O(\log\log q)$. The difficulty  
lies in the small $k$ range such as $k\leq 3m+1$.
\end{section}

\begin{center}
{\bf Acknowledgements}
\end{center}

This work was supported by the Early Career Research Workshop in Coding Theory, Cryptography, and Number Theory held at Clemson University in 2018, under NSF grant DMS-1547399.

\end{document}